\documentclass[a4paper]{amsart}
\usepackage{amssymb}
\usepackage{amsmath}
\usepackage{amsthm}
\usepackage{hyperref}
\usepackage{enumerate}
\usepackage{mathtools}

\newtheorem{Theorem} {Theorem} [section]
\newtheorem{Proposition} [Theorem] {Proposition}
\newtheorem{Lemma} [Theorem] {Lemma}

\newcommand{\m}{\mathcal}

\newcommand{\cB}{{\mathcal B}}

\newcommand{\ff}{{\mathcal F}}

\renewcommand{\phi}{\varphi} 
\newcommand{\gauss}[2]{\genfrac{[}{]}{0pt}{}{#1}{#2}}

\title{Structure of $t$-Intersecting Families of  Vector Spaces}
\author{
Ferdinand Ihringer,
Andrey Kupavskii
}

\begin{document}

\begin{abstract}
We study $t$-intersecting and $t$-cross-intersecting families of $k$-di\-mensional subspaces in finite vector spaces of dimension $n$.
We show that all large $t$-intersecting families admit a governing low-dimensional structure for $n \ge 2k+1$.
This result, together with its cross-intersecting variant, allows us to prove analogues of several classical extremal set-theoretic results. In particular, we determine the intersecting families with the largest diversity, and we establish a Frankl-type degree–diversity result that generalizes the Hilton–Milner theorem.
Our proofs rely on simplification procedures for $t$-intersecting and $t$-cross-intersecting families of subspaces. These procedures are based on the concept of subspace spreadness, a generalization of the classical notion of spreadness for set systems.
\end{abstract}

\maketitle

\section{Introduction}
For a set $X$, write $\binom{X}{k} := \{ F \subset X: |F| = k\}$.
We use $V(n,q)$ to denote the $n$-dimensional vector space over a field of $q$ elements.
Let $\m V_n^k$ ($\m V_n^{\le k}$, $\m V_n$) stand for the family of all $k$-dimensional (all $\le k$-dimensional, all) subspaces in $V(n,q)$. We suppress $q$ from the notation for brevity.
Furthermore, write $\gauss{n}{k} = |\m V_n^k|$ and $[n] = |\m V_n^1|$, see Section~\ref{sec:gauss}
for details.

A family of sets is {\it $t$-intersecting} if any two of its sets intersect in at least $t$ elements. A family of subspaces is {\it $t$-intersecting} if any two of its subspaces intersect in a subspace of dimension at least $t$. We write \emph{intersecting} instead of $1$-intersecting.

In many extremal results for families of sets, the extremal family is governed by a structure of small uniformity. Identifying such a structure is often a key step in proving both extremal and stability results. The main goal of this paper is to develop an analogue of this approach for finite vector spaces. Its key step is to identify a low-dimensional structure in large $t$-intersecting families of subspaces (see Theorem~\ref{thmstruc}). Based on this, we prove analogues of several classical extremal set-theoretic results. We note that our approach also allows to reprove most of the known results on $t$-intersecting families of subspaces. We shall illustrate it below.

\medskip
Next, we give an overview of the literature on the subject.
\subsection{The largest \texorpdfstring{$t$}{t}-intersecting families of sets and vector spaces.}

In~\cite{EKR}, Erd\H{o}s, Ko, and Rado determined the largest $k$-uniform intersecting families of sets:

\begin{Theorem}[EKR Theorem]
 Suppose that $\ff \subset \binom{\{1, \ldots, n \}}{k}$ is an intersecting family
 and $n \geq 2k$. Then $|\ff| \leq \binom{n-1}{k-1}$. For $n \geq 2k+1$, equality holds
 only if $\ff$ is a full star: it consists of all $k$-element sets that contain some fixed element $x$.
\end{Theorem}
Families consisting of sets that contain a fixed element are called {\it stars}. For $n=2k$, there are $2^{\binom{2k-1}{k}}$ extremal examples: for each pair of complementary $k$-sets $F$ and $\bar F$, exactly one of them is included in the family. Thus, there is no hope for a lower-uniformity structural description.

The problem of determining the largest $t$-intersecting $k$-uniform family turned out to be much more difficult. After a series of intermediate results, it was completely solved by Ahlswede and Khachatrian~\cite{AK}. The extremal examples, called Frankl families, have the form
$$\Big\{F\in \binom{\{1,\ldots, n\}}{k}: |F\cap \{1,\ldots, t+2i\}|\ge t+i\Big\},$$
where the optimal choice of $i$ depends on the values of $n,k,t$.

The problem of determining the largest $t$-intersecting family of vector spaces also has a long history. It was solved in most cases by Hsieh~\cite{Hsi1,Hsi2}. In the follow-up papers~\cite{DF83,FW,CP}, the remaining cases were treated and other methods for solving the problem were introduced.
In contrast to the set case, extremal families can also be meaningfully classified when $n=2k$,
see~\cite{GN,Tan}.
Altogether, we have the following result.

\begin{Theorem}[Complete $t$-Intersection Theorem for Vector Spaces]\label{thmtintintro}
 Suppose that $\ff \subset \m V^k_n$ is $t$-intersecting and $n \geq 2k$.
 Then $|\ff| \leq \gauss{n-t}{k-t}$. Equality holds only if
 \begin{itemize}
  \item $\ff = \{ F \in \m V^k_n: X \subset F \}$ for some $X \in \m V^t_n$, or
  \item $\ff = \{ F \in \m V^k_n: F \subset Y \}$ for some $Y \in \m V^{n-t}_n$ and $n=2k$.
 \end{itemize}
\end{Theorem}

The list of extremal examples is considerably simpler for subspaces. We also note that the condition $n\ge 2k$ is not really restrictive: the case $k>n/2$ could be reduced to the case $k<n/2$  using duality. For illustration, we shall give a quick proof of this theorem in most cases in Section~\ref{sec:appl} (see Theorem~\ref{thmtint}).

\subsection{Stability: Hilton--Milner-type results and diversity}

Once an extremal result is obtained, a natural question concerns {\it stability}: what is the largest example that is not contained in the extremal examples? More generally, what is the trade-off between the size of a family and its similarity to the extremal configurations? In the case of intersecting families, this similarity is conveniently captured by the notion of diversity.

For a set family $\m F$, its diversity $\gamma (\m F)$ is the number of sets that do not contain the most popular element. One may think of it as the distance to the closest star. (Note that the extremal example in the EKR theorem has diversity $0$.) For a family $\ff$ of subspaces, the diversity $\gamma(\ff)$ is defined analogously as the number of subspaces not containing the most popular $1$-dimensional subspace.

Answering a question of Erd\H{o}s, Ko, and Rado concerning `non-trivial' intersecting families, Hilton and Milner proved the following.
\begin{Theorem}[Hilton--Milner Theorem~\cite{HM}]
 Suppose that $\ff \subset \binom{\{1, \ldots, n \}}{k}$ is an intersecting family, $k \geq 2$, $n \geq 2k+1$, and $\gamma(\ff) \geq 1$.
 Then $|\ff| \leq \binom{n-1}{k-1}  - \binom{n-k-1}{k-1}+1$. Equality holds only if
 \begin{itemize}
  \item $\ff = \{ F \in \binom{\{ 1, \ldots, n \}}{k}: x \in F, |F \cap Y| \geq 1 \} \cup \{ Y \}$ for some $x \in \{1, \ldots, n \}$ and some $k$-subset $Y$ with $x \notin Y$,
  \item $k=3$ and $\ff = \{ F \in \binom{\{ 1, \ldots, n \}}{k}: |F \cap X| \geq 2\}$ for some $3$-element set $X$.
 \end{itemize}
\end{Theorem}

The vector space analogue of the Hilton--Milner theorem was obtained by Blokhuis et al.~\cite{BBCFMPS} for all $n \geq 2k+1$ except for $(n,q) = (2k+1, 2)$. The remaining case was covered by Wang, Xu, and Zhang~\cite{WXZ}.

\begin{Theorem}[Hilton--Milner Theorem for Vector Spaces]\label{thm:HMVS}
 Suppose that $\ff \subset \m V^k_n$ is an intersecting family, $k \geq 3$, $\gamma(\ff) \geq 1$, and
 $n \geq 2k+1$.
 Then $|\ff| \leq \gauss{n-1}{k-1}  - q^{k(k-1)}\gauss{n-k-1}{k-1}+q^k$. Equality holds only if
 \begin{itemize}
  \item $\ff = \{ F \in \m V^k_n: X \subset F, \dim(F \cap Y) \geq 1 \} \cup \{ F \in \m V^k_n: F \subset X+Y \}$ for some fixed $X \in \m V^1_n$ and $Y \in \m V^k_n$ with $\dim(X \cap Y) = 0$,
  \item $\ff = \{ F \in \m V^k_n: \dim(F \cap X) \geq 2\}$ for some fixed $X \in \m V^3_n$.
 \end{itemize}
\end{Theorem}

In~\cite{DhaeseleerHM}, D'haeseleer solved the Hilton--Milner problem for $t$-intersecting families if either $n \geq 2k+t$ and $q \geq 4$, or $n \geq 2k+t+1$ and $q = 3$. In~\cite{CLWZ}, Cao et al.\ proved the same result under the condition that $n \geq 2k+t+\min(4, 2t)$. Notably, this includes the case $q=2$.

For $n=2k$, it is also natural to ask for a Hilton--Milner-type result for vector spaces, but no such result is currently known. Extending work by Blokhuis, Brouwer, and Sz\H{o}nyi~\cite{BBS}, the first author proved the following in~\cite{Ihr}.

\begin{Theorem}
    Suppose that $\ff \subset \m V^k_{2k}$ is an intersecting family, $k \geq 5$, $q \geq 3$
    and $\gamma(\ff) \geq 1$. Then $|\ff| \leq (1+3q^{-1}) [k] \gauss{2k-2}{k-1}$.
\end{Theorem}

The bound in this theorem asymptotically (as $q \to \infty$) matches the size of the two candidate extremal families. Determining the exact Hilton--Milner result for $n=2k$ remains one of the most interesting open problems in this area.

Frankl~\cite{FranklDeg} obtained a stability result for the EKR theorem that substantially generalizes the Hilton--Milner theorem. It was later proved in diversity form by Kupavskii and Zakharov~\cite{KZ}.

\begin{Theorem}\label{thm:degsets}
 Suppose that $\ff \subset \binom{\{1, \ldots, n \}}{k}$ is an intersecting family
 with $n > 2k > 0$ and $\gamma(\ff) \geq \binom{n-u-1}{k-u}$ for some real $3 \leq u \leq k$.
 Then $|\ff| \leq \binom{n-1}{k-1} + \binom{n-u-1}{k-u} - \binom{n-u-1}{k-1}$.
 Equality holds only if $u$ is an integer and
 \[\ff = \{ F \in \binom{\{1, \ldots, n\}}{k}: x \in F, |F \cap Y| \geq 1 \}
 \cup \{ F \in \binom{\{1, \ldots, n\}}{k}: Y \subset F \}, \]
 for some fixed element $x$ and $Y \in \binom{\{1, \ldots, n\}}{u}$ with $x \notin Y$.
\end{Theorem}

Let us define the vector space analogs of the families from the diversity theorem above.
Let $i \in \{ 2, \ldots, k \}$.
Fix a $1$-dimensional subspace $X$ and an $i$-dimensional subspace $Y_i$ disjoint from $X$.
Then define $\m G_i$ by
\begin{equation*}
    \m G_i = \{F\in \m V_n^k: X\subset F, \dim(F \cap Y_i) \geq 1\}\cup\{G\in \m V_n^k: Y_i\subset X+G\}.
\end{equation*}
One of the main results of this paper is the following vector space analog of Theorem~\ref{thm:degsets}. 

\begin{Theorem}[Frankl Degree-Type Theorem for Vector Spaces]\label{thm:frankldegreetheorem_short}
     Suppose that $\ff \subset \m V^k_n$ is an intersecting family. Assume that for some $i=3,\ldots, k$ we have $\gamma(\ff)\ge q^{k} \gauss{n-i-1}{k-i}$.
    Then  $|\ff|\le |\m G_i|$ under mild conditions on $n$ and $k$, or $q$.
    For $i \geq 4$, the bound is tight only if the family is isomorphic to $\m G_i$;
    for $i=3$, the bound is tight only if the family is isomorphic to $\m G_2$ or $\m G_3$.
\end{Theorem}

A precise statement can be found in Theorem~\ref{thm:frankldegreetheorem}.

\subsection{Further results on intersecting families.}

A question addressed in several recent papers is to determine the intersecting families with the largest diversity. So far, the strongest result is due to Frankl and Wang, who solved the problem for $n > 36k$~\cite{FWDiversity}. We obtain an analogous result for vector spaces.

\begin{Theorem}\label{thm:maxdivshort}
Under mild conditions on $n$ and $k$, or on $q$, the family $\{ F \in \m V^k_n: \dim(F \cap X) \geq 2\}$ for $X \in \m V^3_n$
has the largest diversity among all intersecting families in $\m V^k_n$.
\end{Theorem}
A precise statement can be found in Theorem~\ref{thmlargestdiv}.

Another classical line of research on intersecting families concerns families with a lower bound on `relative' diversity. Specifically, one fixes some $\alpha\in (0,1)$ and asks for the largest intersecting family in which no element is contained in more than an $\alpha$-proportion of sets. Results of this type were obtained by Frankl~\cite{F78} and F\"uredi~\cite{Fu78}.

In this paper, we prove an analogous result for families of subspaces; see Theorem~\ref{thmdiv}. We also show that there are no regular intersecting families of subspaces, see Proposition~\ref{propreg}.

There has been substantial recent work on intersecting families of vector spaces.
For instance, in~\cite{ST}, the authors obtained a product-type inequality on the sizes of two cross-intersecting families of vector spaces; in~\cite{SZ}, the authors proved a degree analogue of the EKR theorem for vector spaces; in~\cite[Theorem 2.41]{Mus}, the author described the three largest examples of intersecting families in vector spaces;
and in~\cite{TT} intersecting families with a biased measure are studied.
One useful result is due to De Boeck, who gives a complete list of all maximal intersecting families of $3$-subspaces of size at least $3q^4+3q^3+2q^2+q+1$.
We give a condensed version of this list in Appendix~\ref{sec:planes}.

\subsection{Low-dimensional structure of \texorpdfstring{$t$}{t}-intersecting families.}

The key ingredient for our results stated above is a low-dimensional approximation of large intersecting families in vector spaces.
For classical intersecting families, low-uniformity kernels (or bases) that define or approximate the family have been used since the early Delta-system method; see the aforementioned papers~\cite{F78, Fu78} or a discussion in a recent survey on Delta-systems by the second author~\cite{Kup78}. A related approach was initiated by Dinur and Friedgut~\cite{DF09}, who studied intersecting families using junta approximations, where a junta is a family defined by a constant number of coordinates.

A more efficient approach to constructing low-uniformity kernels was developed in a recent paper by Zakharov and the second author~\cite{KZ}, where they introduced the peeling-simplification procedure. This approach was further refined in the subsequent work of the second author, see, e.g.,~\cite{Kup77}. In this paper, we adapt this approach to vector spaces by introducing the notion of $r$-subspace spreadness and developing a peeling-simplification procedure based on this notion; see Section~\ref{sec:peeling}.

Before stating the low-dimensional approximation result for intersecting families of vector spaces, let us present a series of examples of $t$-intersecting families of subspaces.
Given a family of subspaces $\m C\subset \m V_n^m$, we define
$$\m C(X):=\{S/X: S\in \m C, X\subset S\}.$$
We also use the notation
$$\m C[X]:=\{S: S\in \m C, X\subset S\},$$
and, more generally, for a family $\m X$,
$$\m C[\m X]:=\bigcup_{X\in \m X}\m C[X].$$
Fix a subspace $X_i^t$ of dimension $t+2i$ and consider the family $\m V(t,i)$ of subspaces of $X_i^t$ of dimension $t+i$. Put $\m K_i^{k,t}:=\m V_n^{k}[\m V(t,i)]$. In words, $\m K_i^{k,t}$ consists of all $k$-subspaces that contain a subspace from $\m V(t,i)$. It is easy to see that $\m V(t,i)$, and consequently $\m K_i^{k,t}$, is $t$-intersecting.
Standard bounds on the Gaussian coefficients, see Lemma~\ref{lem:disj_subspaces} and Lemma~\ref{lem:simple_gauss_bnds}, show that
\begin{align*}
|\m K_i^{k,t}|
&\ge \gauss{t+2i}{t+i} q^{(k-t-i)i}\gauss{n-(t+2i)}{k-t-i}\\
&\ge q^{(t+i)i +(k-t-i)i+(k-t-i)(n-k-i)}\\
&= q^{(k-t)(n-k)-i(n-k-t-i)}.
\end{align*}
Our main technical contribution, which describes the structure of large intersecting families, shows that this bound is essentially tight.

\begin{Theorem}[Structural Theorem for Intersecting Families of Subspaces]\label{thmstruc}
    Fix some positive integers $n,k,t$ and a non-negative integer $s$ such that $k\ge s+t+1$ and $n\ge 2k+s+1$. Assume that $\m F\subset \m V_n^k$ is $t$-intersecting. Then there exists a $t$-intersecting family $\m C_{s+t}$ of subspaces of dimension at most $s+t$ such that the number of subspaces in $\m F$ that do not contain a subspace in $\m C_{s+t}$ is bounded.
Specifically,
\[
      |\m F\setminus \m F[\m C_{s+t}]|\le  C q^{(k-t)(n-k)-(s+1)(n-k-t-s-1)},\]
where
    $C$ can be taken to be $2\big(\frac q{q-1}\big)^{k-t+1}(1 + \frac{q+1}{q^2-q-1})^2$ for $n = 2k+s+1$ and $\big(\frac q{q-1}\big)^{s+3}(1 + \frac{q+1}{q^2-q-1})^2$ for $n\ge 2k+s+2.$
\end{Theorem}

Taking the family $\m K_{s+1}^{k,t}$, it is easy to see that only a negligible portion of it can be captured by any $t$-intersecting family $\m C_{s+t}$ of $(s+t)$-dimensional subspaces. Thus, most of the family $\m K_{s+1}^{k,t}$ must lie in the remainder, showing that the bound on the remainder in the theorem is asymptotically tight. A more thorough discussion of this example from the point of view of the peeling-simplification procedure can be found in Section~\ref{sec:peeling} just before the proof of Theorem~\ref{thmstruc}.

This document is structured as follows. In Section~\ref{sec:gauss}, we describe some basic properties and bounds for the Gaussian coefficients $\gauss{n}{k}$. Section~\ref{sec:peeling} contains the core of the paper: the definition of subspace-spreadness and the description of a peeling procedure for $t$-intersecting and $t$-cross-intersecting families in vector spaces. In Section~\ref{sec:appl}, we apply this method to describe the structure of large intersecting families in greater detail. We conclude with several directions for future work in Section~\ref{sec:con}.

\section{Gaussian coefficients}\label{sec:gauss}


For $q > 1$, define $[n]$ by $[n] = \frac{q^n-1}{q-1}$.
For $n \geq k \geq 0$, define the \textit{Gaussian coefficient} $\gauss{n}{k}$ by
\[
 \gauss{n}{k} = \prod_{i=1}^k \frac{[n-i+1]}{[i]}.
\]
For $k > n$ or $k < 0$, put $\gauss{n}{k} = 0$.
If $n,k$ are integers and $q$ is a prime power, then $|\m V_n^{k}| = \gauss{n}{k}$.
Thus, this agrees with the alternative definition given in the introduction.

Throughout the paper, we will need several estimates for Gaussian coefficients.
This section collects the ones we use.
A standard argument, see for example~\cite[Th. 3.3(1), p. 88]{Hirschfeld2}, gives the following result.
We include a proof for completeness.

\begin{Lemma}\label{lem:disj_subspaces}
  The number of $k$-subspaces of $V(n, q)$ that meet a fixed $m$-subspace in a fixed $\ell$-subspace
  equals
  \[
   q^{(k-\ell)(m-\ell)} \gauss{n-m}{k-\ell}.
  \]
\end{Lemma}

\begin{proof}
    We first prove the assertion for $\ell = 0$.
    Thus, we need to count the $k$-subspaces disjoint from a fixed $m$-subspace $M$.
    We claim that there are precisely $q^{km} \gauss{n-m}{k}$ such subspaces.

    Count ordered bases $\cB = (b_1, \ldots, b_n)$ of $V(n,q)$.
    Clearly,
    \begin{align}
     | \{ \cB: b_i \in M \text{ for } 1 \leq i \leq m \}|
     = \prod_{i=0}^{m-1} (q^m - q^i) \prod_{i=m}^{n-1} (q^n - q^i). \label{eq:bases1}
    \end{align}
    Now fix a $k$-subspace $K$ disjoint from $M$.
    Similarly,
    \begin{align}
     & | \{ \cB: b_i \in M \text{ for } 1 \leq i \leq m, \; b_i \in K \text{ for } m+1 \leq i \leq m+k \}| \notag \\
     & = \prod_{i=0}^{m-1} (q^m-q^i) \prod_{i=0}^{k-1} (q^k-q^i) \prod_{i=m+k}^{n-1} (q^n-q^i). \label{eq:bases2}
    \end{align}
    The quotient of \eqref{eq:bases1} by \eqref{eq:bases2} is exactly the number of
    $k$-subspaces $K$ disjoint from $M$.
    The general case follows by passing to the quotient space modulo a fixed $\ell$-subspace of $M$.
\end{proof}

This implies $q$-analogs of the binomial identity.
\begin{Lemma}\label{lem:qbinid}
 Let $n \geq k \geq 0$. Then
 \[
  \gauss{n}{k} = \gauss{n-1}{k-1} + q^k \gauss{n-1}{k} = q^{n-k} \gauss{n-1}{k-1} + \gauss{n-1}{k}.
 \]
\end{Lemma}
\begin{proof}
 Lemma~\ref{lem:disj_subspaces} with $m=1$ and $\ell=0$ shows the first equality.
 The second equality is the dual statement.
\end{proof}

We now record some simple estimates for Gaussian coefficients.

\begin{Lemma}\label{lem:simple_gauss_bnds}
  Let $a \geq b \geq 0$ and $q \geq 2$.
  Then
  \begin{enumerate}
   \item $q^{a-1} \leq [a] \leq (1+\frac{1}{q-1}) q^{a-1} \leq 2q^{a-1}$,
   \item $q^{b(a-b)} \leq \gauss{a}{b} \leq (1 + \frac{q+1}{q^2-q-1}) q^{b(a-b)} \leq (1+\frac{1}{q-1})^2 q^{b(a-b)}$,
   \item for $q=2$ and $b\ge 4, a\ge 2b-1$ we have $\gauss{a}{b}\ge 2.8 \, q^{b(a-b)}$,
   \item for $q=2$ and $b\ge 4, a\ge 4b-1$ we have $\gauss{a}{b}\ge 3.2 \, q^{b(a-b)}$.
  \end{enumerate}
\end{Lemma}

\begin{proof}
 Part 1 follows from
 \[
 q^{a-1} \leq \frac{q^a-1}{q-1} \leq \frac{q^a}{q-1}.
 \]
 Part 2 is a standard estimate, see~\cite[Lemma 3.5]{NP1995}.

 To prove part 3, note that
 \[
 \gauss{a}{b}= \prod_{i=0}^{b-1}\frac{[a-i]}{[b-i]}
 = \prod_{i=0}^{b-1}\frac{2^{a-i}-1}{2^{b-i}-1}
 = 2^{(a-b)b} \cdot  \frac{\prod_{i=a-b+1}^{a} \big(1-2^{-i}\big)}{\prod_{i=1}^{b} \big(1-2^{-i}\big)}.
 \]
 Provided $a\ge 2b-1$, the last expression is monotone increasing in $a$ for fixed $b$, and also monotone increasing in $b$ along the boundary case $a=2b-1$.
 Therefore, it is enough to check the bound numerically for $b=4$ and $a=7$.
 The proof of part 4 is analogous.
\end{proof}

The proof of Theorem~\ref{thmstruc} requires a bound on a certain sum.
We record the required estimate in the following lemma.

\begin{Lemma}\label{lem:sum_bnd}
    Fix positive integers $n$, $t$, $k$, $q\ge 2$ and a non-negative integer $s$ with $n \geq 2k+s+1$ and $s+t+1 \leq k$.
    Put
    \[
     S := \sum_{i=s+t+1}^k S_i, \text{ where }S_i = \gauss{n-i}{k-i} \gauss{i}{t} [i-t+1]^{i-t}.
    \]
    Then
    \[
    S\le C q^{k(n-k)-(s+t+1)(n-k-s-1)},
    \]
    where
    $C$ can be taken to be $2\big(\frac q{q-1}\big)^{k-t+1}(1 + \frac{q+1}{q^2-q-1})^2$ for $n = 2k+s+1$ and $\big(\frac q{q-1}\big)^{s+3}(1 + \frac{q+1}{q^2-q-1})^2$ for $n\ge 2k+s+2.$
\end{Lemma}

\begin{proof}
Using Lemma~\ref{lem:simple_gauss_bnds}, we obtain
\begin{align*}
S_i &\le \Big(1+\frac 1{q-1}\Big)^{i-t}\Big(1 + \frac{q+1}{q^2-q-1}\Big)^2 q^{(k-i)(n-k)+(i-t)t+(i-t)^2}\\
&=\Big(\frac q{q-1}\Big)^{i-t}\Big(1 + \frac{q+1}{q^2-q-1}\Big)^2 q^{k(n-k)-i(n-k+t-i)}=:a_i.
\end{align*}
If $k=s+t+1$, then the claim is immediate, so we may assume that $k\ge s+t+2$.

Note that the quadratic exponent of $q$ in the expression for $a_i$ is minimized at $i = (n-k+t)/2$.
We estimate the sum $\sum_{i=s+t+1}^k a_i$ by splitting it into two geometric tails.

For $i< \frac{n-k+t}{2}$, we have
\[
\frac{a_i}{a_{i-1}}\le \frac q{q-1}\cdot \frac 1{q^2} < \frac 1q,
\]
and therefore
\[
\sum_{i=s+t+1}^{(n-k+t-1)/2} a_i\le  \frac 1{1-\frac 1q}a_{s+t+1} = \frac{q}{q-1}a_{s+t+1}.
\]
In particular, if $\frac {n-k+t}{2}>k$, which is equivalent to $n>3k-t$, then
\[
S\le \frac{q}{q-1}a_{s+t+1}.
\]
Thus, in what follows we may assume that $n\le 3k-t$.

For $i\ge \frac{n-k+t}{2}$, we similarly have
\[
\frac{a_i}{a_{i+1}}\le \frac 1q,
\]
and hence
\[
\sum_{i=(n-k+t)/2}^k a_i\le \frac 1{1-\frac 1q}a_k = \frac{q}{q-1}a_k.
\]

Combining the two estimates, we obtain
\begin{align*}
S &= \sum_{i=s+t+1}^kS_i\le \frac q{q-1}(a_{s+t+1}+a_k) \\
&= \Big(1 + \frac{q+1}{q^2-q-1}\Big)^2 \cdot
\Big(\Big(\frac q{q-1}\Big)^{s+2}q^{k(n-k)-(s+t+1)(n-k-s-1)} \\
&\hspace{6em}+ \Big(\frac q{q-1}\Big)^{k-t+1}q^{k(n-k)-k(n-2k+t)}\Big).
\end{align*}

A straightforward calculation shows that
\begin{equation}\label{eqcalc1}
k(n-2k+t)-(s+t+1)(n-k-s-1) = (k-(s+t+1))(n-2k-s-1).
\end{equation}
Recall that $k\ge s+t+2$ and $n\ge 2k+s+1$.

If $n = 2k+s+1$, then the right-hand side of \eqref{eqcalc1} is $0$, and therefore
\[
S\le 2\Big(1 + \frac{q+1}{q^2-q-1}\Big)^2\Big(\frac q{q-1}\Big)^{k-t+1}q^{k(n-k)-(s+t+1)(n-k-s-1)}.
\]

Now assume that $n\ge 2k+s+2$.
Then the right-hand side of \eqref{eqcalc1} is at least $k-(s+t+1)$, and hence
\[
a_k\le a_{s+t+1}\cdot\Big(\frac q{q-1}\Big)^{k-(s+t+1)}\cdot q^{-k+s+t+1}<\frac 1{q-1} a_{s+t+1}.
\]
Therefore,
\begin{align*}
S &\le \frac q{q-1}(a_{s+t+1}+a_k) \le \Big(\frac q{q-1}\Big)^2 a_{s+t+1}\\
&=\Big(1 + \frac{q+1}{q^2-q-1}\Big)^2\Big(\frac q{q-1}\Big)^{s+3}q^{k(n-k)-(s+t+1)(n-k-s-1)}.
\end{align*}
This completes the proof.
\end{proof}

\section{Subspace-spreadness and peeling} \label{sec:peeling}

We begin by recalling the following definitions. Given a family of subspaces $\m C\subset \m V_n^k$, we define
$$\m C(X):=\{S/X: S\in \m C, X\subset S\}.$$
We also use the notation
$$\m C[X]:=\{S: S\in \m C, X\subset S\},$$
and, more generally, for a family $\m X$,
$$\m C[\m X]:=\bigcup_{X\in \m X}\m C[X].$$

Given a family of subspaces $\m C\subset \m V_n^k$, we say that $\m C$ is {\it $r$-subspace-spread} if for any $i\in \{1,\ldots, k\}$ and any subspace $X$ of dimension $i$ we have
$$|\m C(X)|< r^{-i}|\m C|.$$ 

\begin{Lemma}\label{lemspreadbound}
    If $\m C\subset \m V^{\le k}_n$ is a family of subspaces of size $>r^k$, then there exists a subspace $X$ such that $|\m C(X)|>1$ and $\m C(X)$ is $r$-subspace-spread.
\end{Lemma}

\begin{proof}
Take an inclusion-maximal subspace $X$ such that $|\m C(X)|\ge r^{-\dim X}|\m C|$. Note that $|\m C(X)|>1$, since otherwise $|\m C|\le r^{\dim X}\le r^k$.

We claim that $\m C(X)$ is $r$-subspace-spread. Suppose that there exists a larger-dimensional subspace $Y$ with $X\subset Y$ such that
$$|\m C(Y)|\ge r^{-\dim Y/X}|\m C(X)|.$$
Then
$$|\m C(Y)|\ge r^{-\dim Y/X}|\m C(X)|\ge r^{-\dim Y}|\m C|,$$
contradicting the maximality of $X$. Here we used that $X\subset Y$, and thus $\dim X+\dim Y/X = \dim Y$. This proves the lemma.
\end{proof}

Now let $\m F\subset \m V_n^k$ be a $t$-intersecting family of subspaces. We describe a peeling-simplification procedure for $\m F$.

Starting from the family $\m C_k:=\m F$, for each $i= k,k-1,\ldots, t$ do the following:
\begin{itemize}
    \item Put $\m C:=\m C_i$.
    \item If there is a subspace $X$ that is a strict subspace of some subspace in $\m C$, and such that $\m C \cup \{X\}$ is $t$-intersecting, then put $\m C:= \m C\setminus \m C[X]\cup \{X\}$. Repeat this step until no such $X$ remains.
    \item Output $\m C_{i-1}:=\m C\cap \m V_n^{\le i-1}$ and $\m W_{i}:=\m C\cap \m V_n^{i}$.
\end{itemize}

\begin{Lemma}\label{lemsimpeel}
    The following holds for each $i=k,k-1,\ldots, t$.
    \begin{enumerate}
        \item We have $\m F= \m F[\m C_{i}]\cup \bigcup_{j=i+1}^k \m F[\m W_j].$
        \item The family $\m C_{i-1}$ is a $t$-intersecting family of subspaces in $\m V_n^{\le i-1}$. The family $\m W_i$ is a $t$-intersecting family of subspaces in $\m V_n^{i}$. Moreover, the family $\m C_{i-1}\cup \m W_i$ does not contain a subfamily $\m G$ and a proper subspace $X$ such that $\m G[X]$ is $[i-t+1]$-subspace-spread.
        \item $|\m W_i|\le \gauss{i}{t}[i-t+1]^{i-t} \leq (1+\frac{1}{q-1})^{i-t+2} q^{i(i-t)} \leq 2^{i-t+2} q^{i(i-t)}.$
    \end{enumerate}
\end{Lemma}

\begin{proof}
Part 1 follows easily by reverse induction on $i$, using that $\m C_i =\m C_i[\m C_{i-1}]\cup \m W_i$.

The first two assertions in part 2 follow directly from the definition of $\m C_{i-1}$ and $\m W_i$. We prove the last assertion indirectly. Suppose that there exist a proper subspace $X$ and a subfamily $\m G\subset \m C_{i-1}\cup \m W_i$ such that $\m G[X]$ is $[i-t+1]$-subspace-spread. By the definition of the peeling procedure, there exists $F\in \m C_{i-1}\cup \m W_i$ such that $\dim(X\cap F) = s<t$. Otherwise, we would have added $X$ to $\m C$ and removed $\m G[X]$.

Since $\m C_{i-1}\cup \m W_i$ is $t$-intersecting, it follows that $F/X$ intersects $S/X$ in a subspace of dimension at least $t-s$ for every $S/X\in \m G(X)$. The dimension of $F/X$ is $i-s$, and thus there are $\gauss{i-s}{t-s}$ subspaces of dimension $t-s$ in $F/X$. On the other hand, by the $[i-t+1]$-subspace-spreadness of $\m G(X)$, any subspace of dimension $t-s$ in $V/X$ is contained in strictly fewer than $[i-t+1]^{-(t-s)}|\m G(X)|$ subspaces from $\m G(X)$. Therefore, the number of subspaces in $\m G(X)$ containing any $(t-s)$-subspace of $F/X$ is strictly smaller than
$$\gauss{i-s}{t-s}[i-t+1]^{-(t-s)}|\m G(X)|=\prod_{j=0}^{t-s-1}\frac{[i-s-j]}{[t-s-j][i-t+1]}|\m G(X)|\le |\m G(X)|,$$
where the last inequality follows, for each $j=0,\ldots, t-s-1$, from
\begin{align*}
\frac{[i-s-j]}{[t-s-j][i-t+1]}
&= \frac{(q^{i-s-j}-1)(q-1)}{(q^{t-s-j}-1)(q^{i-t+1}-1)}\\
&=\frac{q^{i-s-j+1}+1-q-q^{i-s-j}}{q^{i-s-j+1}+1-q^{t-s-j}-q^{i-t+1}}\le 1.
\end{align*}
This, in turn, follows from the fact that $t-s-j,i-t+1\ge 1$ and hence
$$q+q^{i-s-j}\ge q^{t-s-j}+q^{i-t+1}.$$

We conclude that not all subspaces from $\m G(X)$ intersect $F/X$ in a subspace of dimension at least $t-s$, contradicting the $t$-intersection property of $\m C_{i-1}\cup \m W_i$. Thus, no such $\m G$ and $X$ exist.

Finally, we prove part 3. Take any subspace $F\in \m W_i$. Note that it has dimension $i$ and $t$-intersects every other subspace in $\m W_i$. Consider a $t$-subspace $X\subset F$ that is contained in at least a $\gauss{i}{t}^{-1}$-fraction of all subspaces in $\m W_i$. The family $\m W_i(X)$ is a family of $(i-t)$-dimensional subspaces and, by part 2, has no $[i-t+1]$-subspace-spread subfamily of the form $\m W_i(X + Y)$. Therefore, by Lemma~\ref{lemspreadbound}, $|\m W_i(X)|\le [i-t+1]^{i-t}$. Combining these bounds, we obtain
$$|\m W_i|\le \gauss{i}{t}|\m W_i(X)|\le \gauss{i}{t}[i-t+1]^{i-t}.$$

The numerical bounds on $|\m W_i|$ now follow from Lemma~\ref{lem:simple_gauss_bnds}.
\end{proof}


Now we are ready to prove Theorem~\ref{thmstruc}.
%
%

\begin{proof}[Proof of Theorem~\ref{thmstruc}]
We claim that the family $\m C_{s+t}$ obtained from the simplifica\-tion-peeling procedure satisfies the conclusion of the theorem.

By Lemma~\ref{lemsimpeel}, part~2, the family $\m C_{s+t}$ is $t$-intersecting and consists of subspaces of dimension at most $s+t$, as required.

It remains to bound the size of the remainder $\m F\setminus \m F[\m C_{s+t}]$. By Lemma~\ref{lemsimpeel}, part~1, we have
\[
\m F = \m F[\m C_{s+t}] \cup \bigcup_{i=s+t+1}^k \m F[\m W_i],
\]
and therefore
\[
|\m F\setminus \m F[\m C_{s+t}]|
\le \sum_{i=s+t+1}^k |\m F[\m W_i]|.
\]

For each $i\ge s+t+1$, we clearly have $\m F[\m W_i]\subset \m V_n^k[\m W_i]$. Moreover, for a fixed subspace $X\in \m W_i$, the number of $k$-dimensional subspaces containing $X$ is equal to $\gauss{n-i}{k-i}$. Hence,
\[
|\m V_n^k[\m W_i]| \le |\m W_i| \cdot \gauss{n-i}{k-i}.
\]

Using Lemma~\ref{lemsimpeel}, part~3, we obtain
\[
|\m F[\m W_i]| \le |\m V_n^k[\m W_i]|
\le \gauss{n-i}{k-i}\,\gauss{i}{t}[i-t+1]^{i-t}.
\]

Summing over all $i\ge s+t+1$, we get
\begin{align*}
\sum_{i=s+t+1}^k |\m F[\m W_i]|
&\le \sum_{i=s+t+1}^k \gauss{n-i}{k-i}\,\gauss{i}{t}[i-t+1]^{i-t}.
\end{align*}

The required bound now follows from Lemma~\ref{lem:sum_bnd}.
\end{proof}
 Finally, let us analyze the example of $\m K_i^{k,t}$. Recall from the introduction that the family $\m K_i^{k,t} = \m V^k_n[\m V(t, i)]$ is $t$-intersecting and has size of order
$q^{(k-t)(n-k)-i(n-k-t-i)}$.
This shows that the bound in Theorem~\ref{thmstruc} is asymptotically tight once we take $s=i-1$.

To explain this more precisely, note first that the family $\m V(t,i)$ itself cannot be simplified: if $X$ is a strict subspace of some member of $\m V(t,i)$, then adding $X$ destroys the $t$-intersecting property. Indeed, $\m V(t,i)$ consists of all $(t+i)$-subspaces of a fixed $(t+2i)$-subspace, and two such subspaces always intersect in dimension at least $t$; however, a strict subspace of one of them need not have this property with the rest of the family.

Now consider the peeling-simplification procedure applied to $\m K_i^{k,t}$. Since every member of $\m K_i^{k,t}$ contains a member of $\m V(t,i)$, one may think of $\m V(t,i)$ as the core configuration underlying $\m K_i^{k,t}$. Once the procedure reaches this core family, no further simplification is possible. Therefore, in the notation of Theorem~\ref{thmstruc}, when $s=i-1$ the approximating family $\m C_{s+t}=\m C_{t+i-1}$ must be empty. Equivalently,
\[
\m V(t,i)\setminus \m V(t,i)[\m C_{t+i-1}] = \m V(t,i).
\]

Thus, for this example, the remainder is as large as the whole core family $\m V(t,i)$. Since $\m K_i^{k,t}$ is obtained by extending the members of $\m V(t,i)$ to $k$-subspaces, this shows that the bound in the remainder term in Theorem~\ref{thmstruc} is best possible, up to the constant factor. 

\subsection{A variant for cross-\texorpdfstring{$t$}{t}-intersecting families of subspaces}

We say that two families of subspaces $\m A, \m B$ are \emph{cross $t$-intersecting} if for any $F\in \m A$ and $G\in \m B$ we have $\dim(F\cap G)\ge t$.

Assume that $\m A\subset \m V_n^a$ and $\m B\subset \m V_n^b$. In analogy with the one-family case, one can obtain low-dimensional approximations for cross $t$-intersecting families by running a peeling-simplification procedure. However, a key subtlety arises. In the one-family setting, the relevant value of spreadness (and consequently the bounds on the remainder) depends only on the uniformity of the family itself. In the present setting, the situation is asymmetric: when simplifying $\m A$, the relevant spreadness threshold is governed by the uniformity of $\m B$, and vice versa.

This makes it less straightforward to formulate a fully symmetric general statement, but at the same time provides additional flexibility. For instance, one may alternate the peeling procedure between the two families, or first reduce the uniformity of one family to a desired level before proceeding with the other.

Below, we describe the procedure in the case when only $\m B$ is simplified. The same argument can then be applied with the roles of $\m A$ and $\m B$ interchanged.

\medskip

\noindent
\textbf{Peeling for cross $t$-intersecting families.}
Let $\m A\subset \m V_n^a$ and $\m B\subset \m V_n^b$ be cross $t$-intersecting families. We describe a peeling-simplification procedure applied to $\m B$.

Starting from $\m C_b := \m B$, for each $i=b,b-1,\ldots,t$ do the following:
\begin{itemize}
    \item Set $\m C := \m C_i$.
    \item If there exists a subspace $X$ that is a strict subspace of some subspace in $\m C$, and such that $\m C \cup \{X\}$ remains cross $t$-intersecting with $\m A$, then replace
    \[
    \m C:= \m C\setminus \m C[X]\cup \{X\}.
    \]
    Repeat this step until no such $X$ exists.
    \item Define $\m C_{i-1} := \m C\cap \m V_n^{\le i-1}$ and $\m W_i := \m C\cap \m V_n^i$.
\end{itemize}

We state the following without proof, as it follows by a straightforward adaptation of the proof of Lemma~\ref{lemsimpeel}.

\begin{Lemma}\label{lemsimpeel2}
For each $i=b,b-1,\ldots,t$, the following hold:
\begin{enumerate}
    \item We have
    \[
    \m B = \m B[\m C_i]\cup \bigcup_{j=i+1}^b \m B[\m W_j].
    \]
    \item The families $\m C_{i-1}\cup \m W_i$ and $\m A$ are cross $t$-intersecting. Moreover, $\m C_{i-1}\subset \m V_n^{\le i-1}$ and $\m W_i\subset \m V_n^i$. The family $\m C_{i-1}\cup \m W_i$ does not contain a subfamily $\m G$ and a proper subspace $X$ such that $\m G[X]$ is $[a-t+1]$-subspace-spread.\footnote{Note that the spreadness parameter is determined by the uniformity of $\m A$.}
    \item We have
    \[
    |\m W_i|\le \gauss{a}{t}[a-t+1]^{i-t} \le (1+\tfrac{1}{q-1})^{i-t+2} q^{i(a-t)} \le 2^{i-t+2} q^{i(a-t)}.
    \]
\end{enumerate}
\end{Lemma}
\section{Structure of large \texorpdfstring{$t$}{t}-intersecting families of subspaces}
\label{sec:appl}

In this section, we apply Theorem~\ref{thmstruc} and related arguments to derive the results stated in the introduction. We begin with the relative diversity result and the maximum diversity result, both of which follow directly from Theorem~\ref{thmstruc}. For the Frankl-type stability result (Theorem~\ref{thm:frankldegreetheorem_short}), we will require an auxiliary statement on cross-intersecting families (Theorem~\ref{thm:cross}), which is based on similar ideas and may be of independent interest.

\subsection{\texorpdfstring{$t$}{t}-intersecting families}
Here, we give a short proof of Theorem~\ref{thmtintintro} for a wide range of parameters. We note that the parameters are not fully optimized and are chosen so that the proof is the shortest. If one does calculations more carefully, one can get rid of the condition $n-k-t\ge 9$. 
\begin{Theorem}\label{thmtint}
    Let $n,k,t,q$ be positive integers, such that $q\ge2$, $n\ge 2k+2$, $t\le k-1$ and $n-k-t\ge 9.$ Then all largest $t$-intersecting families $\m F\subset \m V_n^k$ are isomorphic to $\m G_X:=\{ F \in \m V^k_n: X \subset F \}$ for some $X \in \m V^t_n$.
\end{Theorem}
\begin{proof}
Take a largest $t$-intersecting family $\m F\subset \m V_n^k.$ If $\m F\subset \m G_X$ for  some $X$ then we are done. Assume not. Apply Theorem~\ref{thmstruc} to $\m F$ with $s=0$ and obtain that there exists $X\in \m V^t_n$ such that $$|\m F\setminus \m F[X]|\le 2^7 q^{(k-t)(n-k)-(n-k-t-1)}\le \frac 12 q^{(k-t)(n-k)}.$$ In the first inequality, we used that $C\le 2^7$ for $n\ge 2k+2$ (the bound comes from the case $q=2$). In the second inequality, we used the assumption $n-k-t\ge 9$. Also, there is a subspace  $Y\in \m F\setminus \m F[X]$, and thus all subspaces in $\m F[X]$ must intersect $Y$ outside $X$. It implies that 
$$|\m F[X]|\le [k]\gauss{n-t-1}{k-t-1}\le \frac{[k][k-t]}{[n-t]}\gauss{n-t}{k-t}\le q^{-2} \gauss{n-t}{k-t}.$$
We have $|\m G_X|=\gauss{n-t}{k-t}\ge q^{(k-t)(n-k)},$ and combining the two displayed inequalities, we get
\begin{align*} |\m F|= |\m F[X]|+|\m F\setminus \m F[X]|\le \frac 12 q^{(k-t)(n-k)}+ q^{-2} \gauss{n-t}{k-t}<|\m G_X|. \qedhere \end{align*}
\end{proof}

\subsection{Large relative diversity}

\begin{Theorem}\label{thmdiv}
    Assume that $\m F\subset \m V_n^k$ is $t$-intersecting and $n\ge 2k+1$. Fix some $\alpha\in (0,1)$ and assume that no $t$-subspace is contained in more than an $\alpha$-fraction of the subspaces in $\m F$. Then
    \[
    |\m F| \le C (1-\alpha)^{-1} q^{(k-t)(n-k) - (n-k-t-1)},
    \]
    where $C$ can be taken to be $2 \left( \frac{q}{q-1} \right)^{k-t+1} \left(1 + \frac{q+1}{q^2-q-1} \right)^2$
    for $n=2k+1$ and $\left( \frac{q}{q-1} \right)^3 \left(1 + \frac{q+1}{q^2-q-1} \right)^2$
    for $n \geq 2k+2$.
\end{Theorem}

\begin{proof}
Apply Theorem~\ref{thmstruc} with $s=0$. Then $\m C_t$ consists of a single $t$-subspace $X$, and
\[
|\m F\setminus \m F[X]|\le C q^{(k-t)(n-k)-(n-k-t-1)},
\]
where $C$ is as specified in Theorem~\ref{thmstruc}.

By assumption, no $t$-subspace is contained in more than an $\alpha$-fraction of the members of $\m F$, and therefore $|\m F[X]|\le \alpha |\m F|$. Hence,
\[
|\m F| = |\m F[X]| + |\m F\setminus \m F[X]| \le \alpha |\m F| + |\m F\setminus \m F[X]|.
\]
Rearranging, we obtain
\[
|\m F| \le (1-\alpha)^{-1} |\m F\setminus \m F[X]|.
\]
Combining this with the bound above yields the desired result.
\end{proof}

The families that are, in a certain sense, the farthest from stars are intersecting families $\ff$
in which all $1$-subspaces have the same popularity.
Such a family $\ff$ is called a \emph{regular intersecting family}.
See~\cite{IK2019} for an investigation of the set case.
\begin{Proposition}\label{propreg}
    There are no regular intersecting families in $\m V_n^k$ for $n\ge 2k$. 
\end{Proposition}
\begin{proof} 
Let $r$ be the number of elements of $\ff$ that contain a given $1$-subspace. Since the family is regular, it is the same for any $1$-subspace.
On the one hand, by standard double counting,
\[
 |\ff| \cdot [k] = [n] \cdot r.
\]
On the other hand, a $k$-subspace of $\ff$ meets all elements of $\ff$, so
\[
 r > |\ff|/[k].
\]
Thus, $[k]^2 > [n]$ which is false for all $n \geq 2k$ and $q \geq 2$.
\end{proof}
\subsection{The largest diversity}

The following theorem gives a precise version of Theorem~\ref{thm:maxdivshort}.
Note that, as an alternative description of $\m G_2$,
\[
 \m G_2 = \{F\in \m V_n^k: \dim(F\cap X)\ge 2\}
\]
for some $3$-dimensional subspace $X$.

\begin{Theorem}\label{thmlargestdiv}
The family $\m G_2$ has the largest diversity among intersecting families of $k$-subspaces whenever one of the following holds:
\begin{enumerate}
  \item $k = 2$, $n \geq 2k$;
  \item $k = 3$, $n \geq 2k+1$;
  \item $q\ge 5$, $k\ge 4$, $n\ge 2k+2$;
  \item $q=4$, $k\ge 6$, $n\ge 2k+2$; 
  \item $q=3$, $k\ge 9$, $n\ge 2k+2$;
  \item $q=3$, $n\ge 2k+3$, $n-k\ge 9$;
  \item $q=2$, $k\ge 4$, $n\ge 2k+3$, $n-k\ge 13$.
\end{enumerate}
Furthermore, if $q \geq 7$, $k \geq 4$, $n \geq 2k+2$,
then any family not contained in $\m G_2$ has diversity at most $2q^{-1} \gamma(\m G_2)$.
\end{Theorem}

\begin{proof}
The cases $k\in\{2,3\}$ follow by inspecting the classification in Section~\ref{sec:planes}.
We therefore assume $k\ge 4$.
Our stability claim that any family not in $\m G_2$ has diversity at most $2q^{-1} \gamma(\m G_2)$ follows from a more careful analysis of our arguments
which we skip for the sake of brevity.

\medskip

\noindent
\textbf{Step 1: Size and diversity of $\m G_2$.}
By Lemma~\ref{lem:disj_subspaces},
\begin{align}
|\m G_2|=\gauss{n-3}{k-3} + \gauss{3}{2}\, q^{k-2}\gauss{n-3}{k-2}.\label{eq:sizeGtwo}
\end{align}
The maximum degree in $\m G_2$ is
\[
\gauss{n-3}{k-3} + \gauss{2}{1}\, q^{k-2}\gauss{n-3}{k-2},
\]
and therefore
\begin{equation}\label{bounddiv}
\gamma(\m G_2)=\left(\gauss{3}{2}-\gauss{2}{1}\right) q^{k-2}\gauss{n-3}{k-2}
= q^{k}\gauss{n-3}{k-2}
\ge q^{(k-2)(n-k)+2}.
\end{equation}

\medskip

\noindent
\textbf{Step 2: Reduction via Theorem~\ref{thmstruc}.}
Let $\ff$ be an intersecting family of maximum diversity. Applying Theorem~\ref{thmstruc} with $t=1$ and $s=1$, we obtain
\[
|\ff\setminus \ff[\m C_2]|
\le C q^{(k-2)(n-k)+2 -(n-k-4)},
\]
where $C$ depends on $n$ as specified in Theorem~\ref{thmstruc}.

In all cases (3)--(6), straightforward estimates yield
\[
\frac{|\ff\setminus \ff[\m C_2]|}{\gamma(\m G_2)} \le \alpha < 1
\]
for suitable $\alpha$ depending on $(q,k,n)$ (specifically, $\alpha<\frac 12$ for cases with $q\ge 4$ and $\alpha\le 0.1$ for cases with $q=3$). In particular, $\ff[\m C_2]\neq\emptyset$ and contributes non-trivially to the diversity.

\medskip

\noindent
\textbf{Step 3: Structure of $\m C_2$.}
Since $\m C_2$ is an intersecting family of subspaces of dimension at most $2$, there are two possibilities.

\smallskip

\emph{Case 1.} All members of $\m C_2$ contain a fixed $1$-subspace $P$.

Then $\gamma(\ff[\m C_2])=0$, and hence
\[
\gamma(\ff)\le |\ff\setminus \ff[\m C_2]| < \alpha \gamma(\m G_2),
\]
a contradiction.

\smallskip

\emph{Case 2.} All members of $\m C_2$ lie in a fixed $3$-subspace $X$.

In what follows, we assume that this holds. Without loss of generality, $X$ is the same as in the definition of $\m G_2$. Then necessarily
\[
|\ff\cap \m G_2| > (1-\alpha)\gamma(\m G_2).
\]

\medskip

\noindent
\textbf{Step 4: Excluding $1$-dimensional intersections.}
We claim that every $F\in \ff$ satisfies $\dim(F\cap X)\ge 2$.

Suppose not, and take $F\in \ff$ with $\dim(F\cap X)\le 1$. Then, by Lemma~\ref{lem:disj_subspaces}, at least $q^2$ of the $2$-subspaces of $X$ are disjoint from $F\cap X$. For each such $Y$, at most $[k]\gauss{n-3}{k-3}$ of its extensions meet $F$. Therefore,
\begin{align*}
|\ff[\m C_2]|
&\le (q+1)\gauss{n-2}{k-2} + q^2 [k]\gauss{n-3}{k-3} \\
&< \beta q^{(k-2)(n-k)+2},
\end{align*}
where straightforward calculations using Lemma~\ref{lem:simple_gauss_bnds} imply that $\beta<1$ (specifically, $\beta=\frac{8}{9}$ for $q=3$ and $\beta<\frac{1}{2}$ for $q\ge 4$).

In each of the cases, we have $\alpha+\beta<1$, which contradicts $\gamma(\ff)\ge \gamma(\m G_2)$.
Thus, all members of $\ff$ satisfy $\dim(F\cap X)\ge 2$, and hence $\ff\subset \m G_2$.
Similarly, one can verify that $\alpha+\beta < \frac{2}{q}$ for $q \geq 7$.

\medskip

\noindent
\textbf{Step 5: The case $q=2$.}
Assume $q=2$, $k\ge 4$, $n\ge 2k+3$, and $n-k\ge 13$.

By \eqref{bounddiv} and Lemma~\ref{lem:simple_gauss_bnds}~part~3,
\[
\gamma(\m G_2)\ge q^k\gauss{n-3}{k-2}\ge 2.8\, q^{(k-2)(n-k)+2}.
\]
On the other hand, Theorem~\ref{thmstruc} with $t=1$ and $s=1$ gives
\[
|\ff\setminus \ff[\m C_2]|
\le 2^8 q^{(k-2)(n-k)+2-(n-k-4)}
\le \tfrac12 q^{(k-2)(n-k)+2},
\]
since $n-k\ge 13$.

Here, we need to more carefully analyze the possible intersections of members of $\ff$ with $X$.

\smallskip

\emph{Case 5a:} There exists $F\in\ff$ with $F\cap X=\{0\}$.

Then every $2$-subspace of $X$ is disjoint from $F\cap X$, and therefore
\[
|\ff[\m C_2]|
\le \gauss{3}{2}\,[k]\gauss{n-3}{k-3}.
\]
Using $\gauss{3}{2}=7$ and Lemma~\ref{lem:simple_gauss_bnds},
\[
|\ff[\m C_2]|< \tfrac78\, q^{(k-2)(n-k)+2}.
\]
Hence
\[
|\ff|
\le |\ff\setminus \ff[\m C_2]|+|\ff[\m C_2]|
< \Big(\tfrac12+\tfrac78\Big) q^{(k-2)(n-k)+2}
< \gamma(\m G_2),
\]
a contradiction.

\smallskip

\emph{Case 5b:} There exist $F,F'\in\ff$ such that
\[
\dim(F\cap X)=\dim(F'\cap X)=1
\quad\text{and}\quad
F\cap X\neq F'\cap X.
\]

Then all members of $\ff[\m C_2]$ either contain the unique $2$-subspace of $X$ spanned by the lines $F\cap X$ and $F'\cap X$, or contain one of the remaining $2$-subspaces of $X$ and in addition intersect either $F$ or $F'$  outside $X$. Consequently,
\[
|\ff[\m C_2]|
\le \gauss{n-2}{k-2} + (\gauss{3}{2}-1)[k]\gauss{n-3}{k-3}.
\]
Since $\gauss{3}{2}-1=6$, Lemma~\ref{lem:simple_gauss_bnds} yields
\[
|\ff[\m C_2]|< \tfrac74\, q^{(k-2)(n-k)+2}.
\]
Thus
\[
|\ff|
\le |\ff\setminus \ff[\m C_2]|+|\ff[\m C_2]|
< \Big(\tfrac12+\tfrac74\Big) q^{(k-2)(n-k)+2}
< \gamma(\m G_2),
\]
again a contradiction.

\smallskip

\emph{Case 5c:} All members $F\in\ff$ with $\dim(F\cap X)=1$ satisfy $F\cap X=P$ for the same $1$-subspace $P$.

Then $\gamma(\ff\setminus \ff[\m C_2])=0$, and hence
\[
\gamma(\ff)\le |\ff[\m C_2]\setminus \ff[\m C_2[P]]|.
\]
The right-hand side is at most $\gamma(\m G_2)$, with equality only if $\ff$ contains all $k$-subspaces intersecting $X$ in a $2$-subspace disjoint from $P$. But then the intersecting property is violated unless there is no $F\in\ff$ with $F\cap X=P$. Therefore, in the extremal case we must have $\ff\subset \m G_2$.

This completes the proof.
\end{proof}

\subsection{A Frankl degree-type result}

In this section, we prove a Frankl degree-type theorem for vector spaces (Theorem~\ref{thm:frankldegreetheorem_short}). We begin with an auxiliary lemma estimating the size and diversity of the families $\m G_i$ (recalled below), then prove a result on cross-intersecting families (Theorem~\ref{thm:cross}), and conclude with the  proof of the following theorem that is a precise version of Theorem~\ref{thm:frankldegreetheorem_short}. 

First, let us recast the extremal construction. Let $i \in \{2, \ldots, k\}$. Fix a $1$-subspace $X$ and an $i$-dimensional subspace $Y_i$ disjoint from $X$, and recall that
\[
\m G_i = \m G_i^{\Delta}\cup \m G_i^\gamma,
\]
where
\begin{align*}
    \m G_i^{\Delta}&:=\{F\in \m V_n^k: X\subset F,\ \dim(F \cap Y_i) \geq 1\},\\
    \m G_i^\gamma&:=\{G\in \m V_n^k: Y_i\subset G+X, \dim(G\cap X) = 0\}.
\end{align*}
Note that by Lemma~\ref{lem:disj_subspaces}, for each $k$-space $G_0$ with $Y_i \subset G_0$,
there are precisely $q^k$ $k$-spaces $G$ with $(G + X) \cap Y_i = G_0$.

\begin{Theorem}\label{thm:frankldegreetheorem}
    Let $n\ge 2k+1$, $k \ge 3$, and assume that $\ff\subset \m V_n^k$ is an intersecting family of subspaces.
    Assume that for some $i=3,\ldots, k$ we have $\gamma(\ff)\ge q^{k} \gauss{n-i-1}{k-i}$.
    Suppose that one of the following conditions holds:
    \begin{enumerate}
     \item $q \geq 2$, $k=3$;
     \item $i \geq 4$, $k \geq 4$, $k \leq \frac12 q \ln q - 9$;
     \item $q \geq 7$, $i \geq 3$, $k \geq 4$, $n \geq 2k+2$;
     \item $q \geq 5$, $i \geq 4$, $k \geq 4$, $n \geq 2k+2$;
     \item $q = 4$, $i \geq 4$, $k \geq 6$, $n \geq 2k+2$;
     \item $q = 3$, $i \geq 4$, $k \geq 9$, $n \geq 2k+4$;
     \item $q = 3$, $i \geq 4$, $n \geq 2k+4$, $n-k \geq 9$;
     \item $q=2$, $i \geq 4$, $k \geq 6$, $n \geq 5k$.
    \end{enumerate}
    Then  $$|\ff|\le |\m G_i|.$$ 
    Moreover, equality holds only if $\m F$ is isomorphic to $\m G_i$
    or, if $i=3$, to $\m G_2$ or $\m G_3$.
\end{Theorem}

The ``2 out of 3'' family $\m G_2$
has size roughly $q^{(k-2)(n-k)+2}$ and maximizes the diversity by  Theorem~\ref{thmlargestdiv}.
The following Lemma~\ref{lem:bndsfranklfams} in particular shows that 
for $i \geq 4$, the families $\m G_i$ are larger.

\begin{Lemma}\label{lem:bndsfranklfams}
Let $n \geq 2k+1$, $k\geq 4$, and $3 \leq i \leq k$. Then:
\begin{enumerate}
\item $|\m G_i| \le (\tfrac{q}{q-1})(1 + \tfrac{q+1}{q^2-q-1}) q^{(k-2)(n-k)+(i-1)} + (1 + \tfrac{q+1}{q^2-q-1}) q^{(k-i)(n-k-1)+k}$;
\item $|\m G_i| \ge [i]\cdot q^{(k-2)(i-1)} \gauss{n-i-1}{k-2} \ge [i] q^{(k-2)(n-k)}$;
\item $|\m G_i| > |\m G_{i-1}|$ for $4 \le i \le k$;
\item $|\m G_2| = |\m G_3|$;
\item $\gamma(\m G_i) = q^{k} \gauss{n-i-1}{k-i} = q^i(\gauss{n-i}{k-i} - \gauss{n-i-1}{k-i-1})$.
\end{enumerate}
\end{Lemma}

\begin{proof}
By Lemma~\ref{lem:disj_subspaces} and Lemma~\ref{lem:qbinid},
\begin{align*}
 & |\m G_i^\Delta| = \gauss{n-1}{k-1} - q^{(k-1)i} \gauss{n-i-1}{k-1},\\
 & |\m G_i^\gamma| = q^k \gauss{n-i-1}{k-i} = q^i \left( \gauss{n-i}{k-i} - \gauss{n-i-1}{k-i-1} \right).
\end{align*}

For Part~1, note that  $|\m G_i| \le |\m G_i^\Delta|+|\m G_i^\gamma| \leq [i]\gauss{n-2}{k-2} + q^k \gauss{n-i-1}{k-i}$, and the claimed bound follows from Lemma~\ref{lem:simple_gauss_bnds}.

For Part~2, count $k$-subspaces $F\in \m G_i^\Delta$ such that $\dim(F\cap Y_i)=1$. There are $[i]$ choices for the $2$-subspace $F\cap (X+Y_i)$, and Lemma~\ref{lem:disj_subspaces} (with $m=i+1$, $\ell=2$) gives $$|\m G_i| \ge |\m G^\Delta_i|\ge [i] q^{(k-2)(i-1)} \gauss{n-i-1}{k-2}.$$ The second inequality follows from Lemma~\ref{lem:simple_gauss_bnds}.

For Part~3, assume $Y_{i-1}\subset Y_i$. Counting as in Part~2, we have  $$|\m G_i^\Delta|-|\m G_{i-1}^\Delta|\ge ([i]-[i-1]) q^{(k-2)(i-1)} \gauss{n-i-1}{k-2},$$ while, by Lemma~\ref{lem:qbinid}, $|\m G_{i-1}^\gamma|-|\m G^\gamma_i|\le q^k\left(\gauss{n-i}{k-i+1} - \gauss{n-i-1}{k-i}\right) = q^{2k-i+1} \gauss{n-i}{k-i+1}$. Comparing these inequalities using Lemma~\ref{lem:simple_gauss_bnds} shows that the former dominates for $i\ge 4$, and hence $|\m G_i|>|\m G_{i-1}|$.

For Part~4, by Lemma~\ref{lem:qbinid},
\begin{align*}
 |\m G_3| &= \gauss{n-1}{k-1} - q^{3(k-1)} \gauss{n-4}{k-1} + q^k \gauss{n-4}{k-3}\\
 &= (q^k+q^{k-1}+q^{k-2}) \gauss{n-3}{k-2} + \gauss{n-3}{k-3}.
\end{align*}
For $|\m G_2|$, by Equation~\eqref{eq:sizeGtwo}, $|\m G_2| = |\m G_3|$.

For Part~5, note that $\gamma(\m G_i) = |\m G_i^\gamma|$.
\end{proof}

In the following theorem, we show that the families $\m G_i^\Delta$ and $\m G_i^\gamma$ are extremal, in a certain natural sense, among cross-intersecting families.

\begin{Theorem}\label{thm:cross}
    Assume that $\m A$, $\m B$ are cross-intersecting families in $\m V_n^k$.
    Moreover, suppose that there exists a $1$-dimensional subspace $X$ such that all subspaces in $\m A$ contain $X$, while all subspaces in $\m B$ avoid $X$. Assume further that for some integer $i$, $3\le i\le k$, we have
    \[
    q^{k} \gauss{n-i-1}{k-i} \le |\m B|\le \delta q^{k} \gauss{n-i}{k-i+1},
    \]
    where $\delta = q^{-3}$ for $i=3$ and $q \leq 5$,
    $\delta = 2q^{-1}$ for $i=3$ and $q \geq 7$, and $\delta = 1$ otherwise.
    Suppose that one of the following conditions holds:
    \begin{enumerate}
     \item $n=2k+1$, $k-i \leq (q-1) \ln q - 9$, $k \geq 4$;
     \item $q \geq 4$, $n \geq 2k+2$, $k \geq 4$;
     \item $q = 3$, $n \geq 2k+4$, $k \geq 4$;
     \item $q = 2$, $n \geq 2k+13$, $n \geq \frac52 k+5$, $k \geq 4$.
    \end{enumerate}
    Then $|\m A|+|\m B| \leq |\m G_i|$, with equality only if $\m A$ and $\m B$ are isomorphic to $\m G_i^\Delta$ and $\m G_i^\gamma$, respectively.
\end{Theorem}

\begin{proof}
Throughout much of the argument, we will work in the quotient of $X$.
For this, we define $\tilde{\m B} = \{ (X+B)/X: B \in \m B\}$. By Lemma~\ref{lem:disj_subspaces},
\[
 \gauss{n-i-1}{k-i} \le |\tilde{\m B}|.
\]
First note that if $\tilde{\m B}$ consists of all $k$-subspaces that contain some fixed $i$-dimensional subspace $Y'$, then we are done. Indeed, the lower bound on $|\tilde{\m B}|$ is precisely the number of such subspaces, and then cross-intersection forces $\m A(X)$ to consist of all $k$-subspaces meeting $Y'$, so that $(\m A,\m B)$ is isomorphic to $(\m G_i^\Delta,\m G_i^\gamma)$ (with $Y$ such that $(X+Y)/X = Y'$).

The proof uses the same ideas as the proof of Theorem~\ref{thmstruc}, but requires substantially more care. We first show that no $i$-subspace lies in too many members of $\tilde{\m B}$. This allows us to ensure that, when we run the peeling procedure, the simplified $\tilde{\m B}$ remains non-trivial. Then we run a peeling procedure with a carefully chosen set of parameters to ensure small remainders and non-triviality of $\tilde{\m B}.$

\medskip

\noindent
\textbf{Step 1: No $i$-subspace is in many elements of $\tilde{\m B}$.}

Suppose that there exists an $i$-dimensional subspace $Y$ contained in at least an $\alpha:=q^{2k-n}$-fraction of the members of $\tilde{\m B}$. We claim that then every subspace in $\m A(X)$ must intersect $Y$.

Indeed, suppose that some $A\in \m A(X)$ is disjoint from $Y$. Then at most $[k-1]\gauss{n-i-2}{k-i-1}$ members of $\tilde{\m B}$ containing $Y$ can intersect $A$. On the other hand,
\begin{align*}
|\tilde{\m B}[Y]|-[k-1]\gauss{n-i-2}{k-i-1}
&\ge q^{2k-n} \gauss{n-i-1}{k-i} - [k-1]\gauss{n-i-2}{k-i-1}\\
&=\Big(q^{2k-n} - \frac{[k-1][k-i]}{[n-i-1]}\Big)\gauss{n-i-1}{k-i}\\
&>\Big(q^{2k-n} - q^{2k-n}\Big)\gauss{n-i-1}{k-i} = 0,
\end{align*}
a contradiction with $\m A(X)$ and $\tilde{\m B}$ being cross-intersecting. Thus every member of $\m A$ must intersect $Y$.

Now let $Y':=\bigcap \tilde{\m B}$. Since we excluded the case when all members of $\tilde{\m B}$ contain an $i$-subspace, we may assume that $\dim(Y'\cap Y)<i$. Put $s:=\dim(Y'\cap Y)$. Then every member of $\m A$ either intersects $Y'\cap Y$, or intersects $Y\setminus Y'$ and, additionally, intersects some member of $\tilde{\m B}$. Consequently,
\begin{align*}
|\m A| &= |\m A(X)|
\le [s]\gauss{n-2}{k-2}+[i][k]\gauss{n-3}{k-3}.
\end{align*}
Also,
\[
|\m B| \le \delta \Big(1+\tfrac{q+1}{q^2-q-1}\Big) q^{(k-i+1)(n-k-1)+k}
= \delta \Big(1+\tfrac{q+1}{q^2-q-1}\Big) q^{(k-i+1)(n-k)+(i-1)}.
\]
By Lemma~\ref{lem:bndsfranklfams},
\[
|\m G_i| \ge [i]q^{(k-2)(i-1)}\gauss{n-i-1}{k-2}
\ge [i] q^{(k-2)(n-k)}.
\]
Therefore,
\begin{align}
1 \le \frac{|\m A|+|\m B|}{|\m G_i|}
&\le
\frac{[s]\gauss{n-2}{k-2}+[i][k]\gauss{n-3}{k-3}+|\m B|}{[i] q^{(k-2)(i-1)}\gauss{n-i-1}{k-2}} \notag\\
&\le \frac{[i-1]}{[i] q^{(k-2)(i-1)}} \prod_{j=1}^{i-1}\frac{[n-j-1]}{[n-k-j+2]} + R \label{eq:ratioA_corrected}\\
&\le q^{-1}\Big(\frac{[n-i]}{q^{k-2}[n-k-i+3]}\Big)^{i-1} + R \le q^{-1}+R, \notag
\end{align}
where
\begin{align*}
 R &= \frac{[i][k]\gauss{n-3}{k-3}+|\m B|}{[i] q^{(k-2)(i-1)}\gauss{n-i-1}{k-2}}
 \leq \frac{[i][k]\gauss{n-3}{k-3}+|\m B|}{[i]q^{(k-2)(n-k)}} \\
 &\leq q^{2k-n-1}(1+\tfrac{q+1}{q^2-q-1})(\tfrac{q}{q-1}) + \delta q^{(3-i)(n-k)+(i-1)}(1+\tfrac{q+1}{q^2-q-1})/[i].
\end{align*}
Inequality \eqref{eq:ratioA_corrected} is a contradiction whenever the right-hand side is less than $1$. This is readily verified in our assumptions.
Hence no $i$-subspace is contained in more than a $q^{2k-n}$-fraction of the members of $\tilde{\m B}$.

\medskip

\noindent
\textbf{Step 2: Simplifying $\m A(X)$.}

We now turn to the simplification argument. We apply the cross-intersecting peeling procedure, first to $\m A(X)$ and then to $\tilde{\m B}$, while controlling the total size of the discarded parts. Let $\m C_i^a$ and $\m C_i^b$ denote the corresponding simplified families. Starting from $\m C_{k-1}^a=\m A(X)$ and $\m C_k^b=\tilde{\m B}$, we perform the sequence
\[
\m C_{k-1}^a\to \m C_{k-2}^a\to \cdots \to \m C_{\lfloor (k-i)/2\rfloor+1}^a,
\qquad
\m C_k^b\to \m C_{k-1}^b\to \cdots \to \m C_{\lfloor (k+i)/2\rfloor}^b.
\]

For the simplification of $\m A(X)$, define, for any real $j$, $1 \leq j \leq k-1$,
\[
a_j := [k]^j\gauss{n-j-1}{k-j-1}.
\]
Then, by Lemma~\ref{lemsimpeel2} part~3, in the $(a,j)$-steps for $j=k-1,k-2,\ldots,\lfloor (k-i)/2\rfloor+2$ we have
\[
|\m V_{n-1}^{k-1}[\m W_j^a]|\le a_j.
\]
Moreover, using $\frac{x+1}{y+1} < \frac{x}{y}$ for $x > y$,
\[
\frac{a_{j-1}}{a_j}
= \frac{[n-j]}{[k][k-j]}
> \frac{q^{n-k}}{[k]}
> \frac12 q^{n-2k+1}\ge q.
\]
Similarly, $a_{j-\frac12} \geq a_{j}$.
Bounding $\sum a_j$ by the sum of a geometric progression, we get
\begin{align*}
\sum a_j
&\le \frac{1}{1-1/q}\,a_{ \lfloor (k-i)/2 \rfloor + 2} \leq \frac{1}{1-1/q}\,a_{ (k-i+3)/2}\\
&= (\tfrac{q}{q-1})[k]^{ (k-i+3)/2}\gauss{n- (k-i+5)/2}{k- (k-i+5)/2}\\
 &< (\tfrac{q}{q-1})^{ (k-i+9)/2} q^{((k-i+3)/2)(k-1) + (k- (k-i+5)/2)(n-k)}\\
&< (\tfrac{q}{q-1})^{(k-i+9)/2} q^{(k-i)/2\cdot(2k-n-1)+(k-\frac52)(n-k) + \frac32 (k-1)}.
\end{align*}
Hence $\sum a_j<\beta_1|\m G_i|$ provided
\[
\Big(\tfrac{q}{q-1}\Big)^{(k-i+9)/2} q^{(k-i)/2\cdot(2k-n-1)+\frac32 (k-1)}/q^{\frac12 (n-k)+(i-1)} \le \beta_1.
\]
For $n \geq 2k+2$, the left-hand side is increasing in $i$, so we only need to check $i=k$.
We get the following:
\begin{align*}
 \begin{cases}
  (\tfrac{q}{q-1})^{(k-i+9)/2} q^{-1} & \text{ if } n=2k+1, \\
  (\tfrac{q}{q-1})^{9/2} q^{(-n+2k-1)/2} & \text{ if } n \geq 2k+2.
 \end{cases}
\end{align*}
This is satisfied for $\beta_1=0.34$ if $k-i \le (q-1) \ln q - 9$;
for $\beta_1=0.46$ for $n\geq 2k+2$ and $q \geq 4$;
for $\beta_1=0.11$ for $n\geq 2k+2$ and $q \geq 7$;
for $\beta_1=0.4$ for $n \geq 2k+4$ and $q=3$;
for $\beta_1=0.18$ for $n \geq 2k+13$ and $q=2$.

\medskip

\noindent
\textbf{Step 3: Simplifying $\tilde{\m B}$.}

For the simplification of $\tilde{\m B}$, define, for any real number $j$, $1 \leq j \leq k$,
\[
b_j := [\lfloor (k-i)/2\rfloor+1]^j\gauss{n-j-1}{k-j}.
\]
Then, by Lemma~\ref{lemsimpeel2} part~3, in the $(b,j)$-steps for $j=k,k-1,\ldots,\lfloor (k+i)/2\rfloor+1$ we have
\[
|\m V_n^k[\m W_j^b]|\le b_j.
\]
If $i=k$, then there is no simplification to perform. Otherwise, using $\frac{x+1}{y+1} < \frac{x}{y}$ for $x > y$,
\[
\frac{b_j}{b_{j+1}}
= \frac{[n-j-1]}{[\lfloor (k-i)/2\rfloor+1][k-j]}
> \frac{q^{n-k-1}}{[\lfloor (k-i)/2\rfloor+1]}
> \frac12 q^{n-k-1-\lfloor (k-i)/2\rfloor}
> q.
\]
Similarly, $b_{j} \geq b_{j+\frac12}$.
Hence, we get
\begin{align*}
\sum b_j
&\le \frac{1}{1-1/q}\,b_{\lfloor (k+i)/2\rfloor+1} \leq \frac{1}{1-1/q}\,b_{(k+i+1)/2}\\
&= (\tfrac{q}{q-1})[\lfloor (k-i)/2\rfloor+1]^{(k+i+1)/2}\gauss{n-(k+i+3)/2}{k- (k+i+1)/2}\\
    &< (\tfrac{q}{q-1})^{(k+i+7)/2 } q^{\lfloor (k-i)/2\rfloor \cdot (k+i+1)/2 + (k-(k+i+1)/2 ) (n-k-1)}\\
&\leq (\tfrac{q}{q-1})^{(k+i+7)/2} q^{(k-i)/2\cdot(k+i+1)/2 + (k-(k+i+1)/2)(n-k-1)}.
\end{align*}
Therefore $\sum b_j<(1-q^{-1}) q^{(k-i)(n-k-1)}\le (1-q^{-1}) |\tilde{\m B}|$ provided
\[
\Big(\tfrac{q}{q-1}\Big)^{(k+i+7)/2} q^{(k-i)/2\cdot(k+i+1)/2 - ((k-i+1)/2)(n-k-1)}\le 1-q^{-1}.
\]
Since the left-hand side is increasing in $i$, it is enough to check the inequality for $i=k-1$, which gives
\[
\Big(\tfrac{q}{q-1}\Big)^{k+3}q^{-n+\frac32 k +1} < 1-q^{-1}.
\]
This holds for $n\ge 2k+1$ and $q\ge 4$; $n\geq 2k+2$ and $q=3$; $n\ge \frac52 k+5$ and $q=2$.

In short, the bound on the remainder above implies that we threw away at most $1-q^{-1}$ of subspaces of $\tilde{\m B}$ while doing the peeling procedure. Together with the fact that no $i$-subspace is contained in more than a $q^{2k-n}$-fraction of members of $\tilde{\m B}$ (which was shown in the beginning of the proof), this implies that the resulting family $\m C^b_{\lfloor (k+i)/2\rfloor}$ is still non-trivial, i.e., no $i$-subspace is contained in all of its members.

Set
\[
\m A':=\m A(X)[\m C_{\lfloor (k-i)/2\rfloor+1}^a],
\qquad
\m C^b:=\m C_{\lfloor (k+i)/2\rfloor}^b.
\]
By the definition of peeling, $\m A'$ and $\m C^b$ remain cross-intersecting. Put $Y:=\bigcap \m C^b$.
By the argument above, $\dim Y=s\le i-1$. The number of members of $\m A'$ that intersect $Y$ is at most
\[
[s]\gauss{n-2}{k-2},
\]
and, exactly as in \eqref{eq:ratioA_corrected}, this is at most
\[
q^{-1}|\m G_i|.
\]

\medskip

\noindent
\textbf{Step 4: Bounding $|\m A| + |\m B|$.}

It remains to bound the members of $\m A'$ that avoid $Y$ but still cross-intersect $\m C^b$. Take any $F\in \m C^b$. Since $\dim F\le \frac{k+i}{2}-1$, any $G\in \m A'$ avoiding $Y$ must contain some $1$-subspace $S\subset F$. As $S$ is not contained in the common intersection of $\m C^b$, there exists $F'\in \m C^b$ avoiding $S$. Hence $G$ must also contain some $1$-subspace of $F'$. Therefore the number of such $G$ is at most
\[
\Big[\frac{k+i}{2}\Big]^2\gauss{n-3}{k-3}
\le \Big(\tfrac{q}{q-1}\Big)^2\Big(1+\tfrac{q+1}{q^2-q-1}\Big) q^{(k-3)(n-k)+k+i-2}.
\]
Consequently,
\[
|\{G\in \m A': G\cap Y=\{0\}\}|
\le \frac{(\tfrac{q}{q-1})^2(1+\tfrac{q+1}{q^2-q-1})}{[i]} q^{(2k-n-1)+(i-1)} |\m G_i|
=: \beta_3 |\m G_i|,
\]
where we may take $\beta_3=0.03$ for $n=2k+1$ and $q\ge 7$; $\beta_3=0.11$ for $n\ge 2k+2$ and $q\ge 3$; $\beta_3=0.01$ for $n\ge 2k+10$ and $q\ge 2$.

Finally,
\[
\frac{|\m B|}{|\m G_i|}
\le \delta \frac{1+\tfrac{q+1}{q^2-q-1}}{[i]} q^{(3-i)(n-k)+(i-1)}
=: \beta_4,
\]
where we may take $\beta_4=0.01$ for $i \geq 4$; $\beta_4=2/7$ for $i=3$ and $q=2$;
$\beta_4 = 0.05$ for $i=3$ and $q \in \{ 3, 4, 5 \}$; $\beta_4 = 0.3$ for $i=3$ and $q \geq 7$.

Collecting the discarded part from $\m A$, the members of $\m A'$ intersecting $Y$, the members of $\m A'$ avoiding $Y$, and finally the contribution of $\m B$, we obtain
\[
|\m A|+|\m B|
\le (\beta_1+q^{-1}+\beta_3+\beta_4)|\m G_i|.
\]
In all admissible ranges, this coefficient is strictly smaller than $1$.
This proves the theorem.
\end{proof}

Now we are ready to prove our Frankl degree-type theorem for vector spaces.

\begin{proof}[Proof of Theorem~\ref{thm:frankldegreetheorem}]
   We start by treating the case $i=3$, $q\ge 7,k\ge 4, n\ge 2k+2$. Then we will treat all cases with $q\ge 3$ and $i\ge 4.$ We will treat the case $q=2$ separately at the end.
    The case $k=3$ follows from Appendix~\ref{sec:planes}.

Assume that $i=3$, $q\ge 7,k\ge 4, n \ge 2k+2$. Take a $1$-subspace $X$ such that $\m F[X]$ is the largest. If $|\ff \setminus \ff[X]| > 2 q^{k-1} \gauss{n-3}{k-2}$,
then the stability part of Theorem~\ref{thmlargestdiv} implies that $\ff$ is isomorphic to $\m G_2$. If $|\ff \setminus \ff[X]| < 2 q^{k-1} \gauss{n-3}{k-2}=\delta q^{k}\gauss{n-3}{k-2},$ then Theorem~\ref{thm:cross} is applicable and we conclude that $\ff$ must be isomorphic to  $\m G_3.$

In what follows, we assume that $i\ge 4.$


\medskip

\noindent
\textbf{Step 1: $\gamma(\ff)$ is small compared to $|\ff|$.}
    We want to bound $\gamma(\ff)$ and show that there is a most popular $1$-dimensional
    subspace $X$. More specifically, we will show $\gamma(\ff) < \frac{41}{100} |\ff|$
    for $q \geq 3$ (and, later, $\gamma(\ff) < \frac13 |\ff|$ for $q=2$).
    We distinguish the cases $n=2k+1$ and $n \geq 2k+2$.

    If $n=2k+1$, then we apply Theorem~\ref{thmdiv} with $\alpha=3/5$ and find
    \begin{align*}
     |\ff| &\leq 5 \left(\frac{q}{q-1}\right)^{k} \left( 1 + \frac{q+1}{q^2-q-1} \right)^2 q^{k(k-1)}.
    \end{align*}
    Comparing this to the lower bound on $|\m G_i| \leq |\ff|$ in Lemma~\ref{lem:bndsfranklfams}
    gives the condition
    \[
     [i] > 5q^2 \left(\frac{q}{q-1}\right)^{k} \left( 1 + \frac{q+1}{q^2-q-1} \right)^2.
    \]
    This is implied by $k \leq \frac12 q \ln(q) - 9$.

    For any case $n \geq 2k+2$, the choice of conditions
    is such that we can apply Theorem~\ref{thmlargestdiv}.
    Here we want to show
    \[
     \gamma(\ff) < \frac{41}{100} |\ff|.
    \]
    Equation~\eqref{bounddiv}, with Lemma~\ref{lem:simple_gauss_bnds},
    implies that $\gamma(\ff) \leq q^{(k-2)(n-k)+2} \allowbreak \left( 1 + \frac{q+1}{q^2-q-1} \right)$.
    By Lemma~\ref{lem:bndsfranklfams}, $|\m F| \geq |\m G_i| \geq [i] q^{(k-2)(n-k)}$. Thus,
    we obtain the condition
    \[
        q^2 \left( 1 + \frac{q+1}{q^2-q-1} \right) < \frac{41}{100} (1+q^{-1}+q^{-2}+q^{-3}) q^{i-1}.
    \]
    It is straightforward to check that this holds for all $q \geq 3$.
    Thus, there is a most popular $1$-dimensional subspace $X$.

    \medskip

\noindent
\textbf{Step 2: Apply Theorem~\ref{thm:cross}.}
    The theorem follows from Theorem~\ref{thm:cross}, applied to $\ff[X]$ and $\ff\setminus \ff[X]$:

Clearly, $|\ff \setminus \ff[X]| = \gamma(\ff) \geq q^{k} \gauss{n-i-1}{k-i}$.
If $|\ff \setminus \ff[X]| \leq q^{k-3} \gauss{n-3}{k-2}$, then we
can apply Theorem~\ref{thm:cross} and see that $\ff$ is isomorphic to some family
$\m G_j$ with $j \geq i$. By Lemma~\ref{lem:bndsfranklfams},
$|\m G_i| > |\m G_{i-1}|$ and we are done.

It remains the case that $|\ff \setminus \ff[X]| > q^{k-3} \gauss{n-3}{k-2}$. Throw away some of $\ff \setminus \ff[X]$ such that Theorem~\ref{thm:cross}
becomes applicable with $i=3$. Then
\[
 \tfrac{59}{100} |\ff| \leq |\ff[X]| \leq [3] \gauss{n-2}{k-2} \leq (1+q^{-1}+q^{-2}) (1+\tfrac{q+1}{q^2-q-1}) q^{(k-2)(n-k)+2}.
\]
By Lemma~\ref{lem:bndsfranklfams},
\[
 (1+q^{-1}+q^{-2}+q^{-3}) q^{(k-2)(n-k)+i-1} \leq |\ff|.
\]
This is a contradiction for all $q \geq 3$.

\medskip

\noindent
\textbf{Step 3: The case $q=2$.}
Now let us cover the differences for $q=2$.
By Lemma~\ref{lem:simple_gauss_bnds}, $|\ff| \geq |\m G_i| \geq 3.2 [i] q^{(k-2)(n-k)}$.
Thus, to show $\gamma(\ff) < \frac13 |\ff|$, we need to show
\[
    q^2 \left( 1 + \frac{q+1}{q^2-q-1} \right) \leq \frac13 \cdot 3.2 \cdot (1+q^{-1}+q^{-2}+q^{-3}) q^{i-1}
\]
which is true for $q=2$ and all $i \geq 4$.
Then, similar to before, it is easy to check that the following does not hold:
\begin{align*}
 & 3.2 \cdot (1+q^{-1}+q^{-2}+q^{-3}) q^{(k-2)(n-k)+i-1} \leq |\ff| \\
 & \leq  \tfrac32 |\ff[X]| \leq (1+q^{-1}+q^{-2}) (1+\tfrac{q+1}{q^2-q-1}) q^{(k-2)(n-k)+2}.
\end{align*}
This covers the case $q=2$.

Let us conclude by reviewing the conditions.
For $n=2k+1$, we used $k \leq \frac12 q \ln q - 9$ in Step 1.
All other constraints follow from Theorem~\ref{thmlargestdiv} or Theorem~\ref{thm:cross}.
This completes the proof of the theorem.
\end{proof}

\section{Future work}\label{sec:con}
Below we list several directions for future research.

\subsection*{\texorpdfstring{The case $n=2k$}{The case n=2k}}

For intersecting families of vector spaces, the case $n=2k$ is arguably the most interesting.
Unlike the setting of sets, it is still feasible to characterize large intersecting families.
However, in contrast to the regime $n \geq 2k+1$, extremal families in this case naturally come in dual pairs.

For instance, for $t < k/2$, there are two types of $t$-intersecting families of size $\gauss{n-t}{k-t} = \gauss{2k-t}{k-t}$:
\begin{itemize}
 \item $\{ F \in \m V^k_{2k}: X \subset F \}$ for some fixed $X \in \m V^t_{2k}$,
 \item $\{ F \in \m V^k_{2k}: F \subset X \}$ for some fixed $X \in \m V^{2k-t}_{2k}$.
\end{itemize}
Even in the $1$-intersecting case, only partial Hilton--Milner-type results are known, cf. the Introduction and~\cite{BBS,Ihr}.

It would be desirable to develop a version of the peeling procedure for subspaces, and of Theorem~\ref{thmstruc}, that also applies in the regime $n=2k$. Such a result could also shed light on borderline cases such as $n=2k+1$, where the currently known bounds on $q$ are somewhat unsatisfactory.

\subsection*{Affine spaces}

Intersecting families in affine spaces have been studied, for example, in~\cite{CLWZ,DhaeseleerHM,GLM} in the context of Hilton--Milner-type results.
There are two standard models of an $n$-dimensional affine space $AG(n, q)$.
In one model, we identify $AG(n, q)$ with $V(n, q)$, in which case a $k$-subspace has the form $M+v$, where $M$ is a $(k-1)$-subspace of $V(n, q)$.
In the second model, we identify $AG(n, q)$ with $V(n+1, q)$ minus a fixed hyperplane $H$, so that the $k$-subspaces of $AG(n, q)$ correspond to the $k$-subspaces of $V(n+1, q)$ not contained in $H$.
From this perspective, it is clear that Theorem~\ref{thmstruc} extends to affine spaces as well.

It would be interesting to determine whether there exist problems on intersecting families for which affine spaces exhibit behavior genuinely different from that of vector spaces.

\subsection*{Bilinear forms}

Let $n \geq 2k$.
Two $k \times (n-k)$ bilinear forms $x, y$ over the field with $q$ elements are called \textit{intersecting} if $x-y$ does not have full rank.
Bilinear forms can also be realized as a subset of $V(n, q)$ as follows:
fix an $(n-k)$-subspace $M$ and consider all $q^{k(n-k)}$ $k$-subspaces disjoint from $M$.
Each such $k$-subspace $X$ corresponds to a bilinear form $x$, and for two such subspaces $X, Y$ with corresponding forms $x, y$, we have $\dim(X \cap Y) = i$ if and only if $x-y$ has rank $k-i$.
See \S9.5A~\cite{BCN} for an explicit description of this correspondence.

The largest intersecting families of bilinear forms were determined by Huang~\cite{Huang}, and a Hilton--Milner-type result was recently obtained by Cao et al.~\cite{CLW}.
Given the close connection to intersecting families of vector spaces, Theorem~\ref{thmstruc} applies in this setting as well.
As for affine space, it would be interesting to identify whether there are scenarios that distinguish intersecting families of bilinear forms from those of subspaces.
In particular, hypercontractivity exists for bilinear forms, see~\cite{EKL}, so it is the natural setting for this machinery.

\section*{Acknowledgements} This research was supported by National Key R\&D Program of China under grant number 2025YFA1017700.

\appendix

\section{Large families for \texorpdfstring{$k=3$}{k=3}} \label{sec:planes}

Later, for numerical estimates, it is convenient to exclude small $k$.
In this section we will list all large intersecting families for $k=2,3$.

For $k=2$, there are only two maximal intersecting families:
\begin{enumerate}
 \item All $2$-subspaces through a fixed $1$-subspace. Size: $[n-1]$.
 \item All $2$-subspaces in a fixed $3$-subspace. Size: $[3]$.
\end{enumerate}

For $n \geq 6$ and $k=3$,
De Boeck gives a complete list of (maximal) intersecting families $\m F$
of $3$-subspaces of size at least $3q^4+3q^3+2q^2+q+1$ in~\cite{DeB}.
In the following we give a concise version of that list.
See~\cite{DeB} for details.
For the sake of brevity, we write ``$F:$'' instead
of ``$F \in \m V^3_n:$''
in the definition of $\m F$.
\begin{enumerate}
 \item Let $P \in \m V^1_n$.
%
%
 Put $\m F = \{ F: P \subset F \}$.

 \smallskip

 Size: $\gauss{n-1}{2}$.
 \item Let $P \in \m V^1_n$, $Y \in \m V^3_n$, $\dim(P \cap Y) = 0$.
%
%
 Put $\m F = \{ F: P \subset F, \dim(F \cap Y) \geq 1 \} \cup \{ F: Y \subset P+F \}$.

 \smallskip

 Size: $[3][n-2] - q[2]$.
 \item Let $X \in \m V^3_n$.
%
%
 Put $\m F = \{ F: \dim(F \cap X) \geq 2 \}$.

 \smallskip

 Size: $[3][n-2] - q[2]$.
 \item Let $P \in \m V^1_n$, $X \in \m V^3_n$, $Y \in \m V^5_n$, $P \subset X \subset Y$.

 \smallskip

 Put $\m F = \{F: P \subset F \subset Y \} \cup \{ F: P \subset F, \dim(F \cap X) = 2 \} \cup \{ F: F \subset Y, \dim(F \cap X) = 2 \}$.

 \smallskip

 Size: $[2][n-2] + (2q^4+q^3-q)$.
 \item Let $L \in \m V^2_n$, $Y \in \m V^5_n$, $L \subset Y$.

 \smallskip

 Put $\m F = \{ F: L \subset F \} \cup \{ F: F \subset Y, \dim(F \cap L) = 1 \}$.

 \smallskip

 Size: $[n-2] + q^2[2][3]$.
 \item Let $P_1, P_2 \in \m V^1_n$, $L \in \m V^2_n$, $X_1, X_2 \in \m V^3_n$, $Y \in \m V^4_n$, $Z_1, Z_2 \in \m V^5_n$, where $P_1, P_2 \subset L = X_1 \cap X_2$,
 $X_1, X_2 \subset Y = Z_1 \cap Z_2$.

 \smallskip

 Put $\m F = \{ F: L \subset F \} \cup \{ F: F \subset Y \} \cup \{ F: P_i \subset F \subset Z_1, \dim(F \cap X_i) \geq 2, i \in \{ 1, 2\} \} \cup \{ F: P_i \subset F \subset Z_2, \dim(F \cap X_j) \geq 2, \{i,j \} = \{ 1, 2\} \}$.

 \smallskip

 Size: $[n-2] + 5q^3 + q^2$.
 \item Let $P_1, P_2, P_3, Q_1, Q_2, Q_3, Q_4 \in \m V^1_n$, $L, L_1, L_2, L_3, \overline{L}_1, \overline{L}_2, \overline{L}_3 \in \m V^2_n$, $X \in \m V^3_n$, $Y \in \m V^5_n$, where
 $L, X \subset Y$, $\dim(L \cap X) = 0$, $P_1, P_2, P_3 \subset L$,
 $Q_1, Q_2, Q_3, Q_4 \subset X$ non-coplanar (i.e., any three span a distinct $3$-subspace),
 $L_1 = Q_1 + Q_2$, $\overline{L}_1 = Q_3+Q_4$,
 $L_2 = Q_1 + Q_3$, $\overline{L}_2 = Q_2+Q_4$,
 $L_3 = Q_1 + Q_4$, $\overline{L}_3 = Q_2+Q_3$.

 \smallskip

 Put $\m F = \{ F: L \subset F \} \cup \{ F: P_i \subset F \subset L+L_i, i \in \{1,2,3\}\}
 \cup \{ F: P_i \subset F \subset L+\overline{L}_i, i \in \{1,2,3\}\}$.

 \smallskip

 Size: $[n-2] + 6q^2$.
 \item Let $P_1, P_2 \in \m V^1_n$, $L \in \m V^2_n$, $X_1, X_2 \in \m V^4_n$,
 where $P_1, P_2 \subset L = X_1 \cap X_2$.

 \smallskip

 Put $\m F = \{ F: L \subset F \} \cup \{ F: P_1 \subset F, \dim(F \cap X_1), \dim(F \cap X_2) \geq 2\} \cup \{ F: P_2 \subset F \subset X_1 \} \cup \{ F: P_2 \subset F \subset X_2 \}$.

 \smallskip

 Size: $[n-2] + q^4 + 2q^3 + 3q^2$.
 \item Let $P_1, P_2 \in \m V^1_n$, $L, R_1, \ldots, R_{q+1}, R_1', \ldots, R_{q+1}' \in \m V^2_n$, $X \in \m V^4_n$,
 where $\dim(L \cap X) = 0$, $P_1, P_2 \subset L$, $R_1, \ldots, R_{q+1}, R_1', \ldots, R_{q+1}' \subset X$, $\dim(R_i \cap R_j') = 1$ for all $i,j \in \{ 1, \ldots, q+1 \}$,
 $\dim(R_i \cap R_j) = 0$ and $\dim(R_i' \cap R_j') = 0$ for all distinct $i,j \in \{ 1, \ldots, q+1 \}$ (in projective geometry, $\{ R_1, \ldots, R_{q+1} \}$ is called a \textit{regulus} and $\{ R_1', \ldots, R_{q+1}' \}$ is its \textit{opposite regulus}).

 \smallskip

 Put $\m F = \{ F: L \subset F \} \cup \{ F: P_1 \subset F \subset L+R_i, 1 \leq i \leq q+1\}
 \cup \{ F: P_2 \subset F \subset L+R_i', 1 \leq i \leq q+1\}$.

 \smallskip

 Size: $[n-2] + 2q^2[2]$.
 \item Let $X \in \m V^5_n$.
 Put $\m F = \{ F: F \subset X \}$.

 \smallskip

 Size: $\gauss{5}{3}$.
 \item Let $L_1, L_2 \in \m V^2_n$, $X \in \m V^4_n$, $Y_1, Y_2 \in \m V^5_n$,
 where $L_1, L_2 \subset X = Y_1 \cap Y_2$ and $\dim(L_1 \cap L_2) = 0$.

 \smallskip

 Put $\m F = \{ F: F \subset X \} \cup \{ F: L_1 \subset F \subset Y_1 \}
 \cup \{ F: L_2 \subset F \subset Y_1 \} \cup \{ F: F \subset Y_2, \dim(F \cap L_1) \geq 1, \dim(F \cap L_2) \geq 1 \}$.

 \smallskip

 Size: $q^4+3q^3+4q^2+q+1$.
\end{enumerate}

\end{document}